\documentclass[12pt]{amsart}

\usepackage{amscd, amsfonts, amsmath, amssymb, amsthm, enumerate, flexisym, float, graphicx, hyperref, pdfpages,tikz, mathtools,comment, pinlabel, subcaption, tabu, tikz, tikz-cd, url, wasysym}
\usepackage[autostyle]{csquotes}
\usepackage[all,cmtip]{xy}
\usepackage{tikz-cd}
\usetikzlibrary{cd, arrows, calc, knots, patterns, positioning, shapes.callouts, shapes.geometric, shapes.symbols}
\tikzset{>=stealth', arrow/.style={->}}
\makeatletter
\@namedef{subjclassname@2020}{Mathematics Subject Classification \textup{2020}}

\newsavebox{\@brx}
\newcommand{\llangle}[1][]{\savebox{\@brx}{\(\m@th{#1\langle}\)}%
  \mathopen{\copy\@brx\kern-0.5\wd\@brx\usebox{\@brx}}}
\newcommand{\rrangle}[1][]{\savebox{\@brx}{\(\m@th{#1\rangle}\)}%
  \mathclose{\copy\@brx\kern-0.5\wd\@brx\usebox{\@brx}}}

\makeatother

\newtheorem{theorem}{Theorem}[section]
\newtheorem{corollary}[theorem]{Corollary}

\newtheorem{prop}[theorem]{Proposition}

\theoremstyle{definition}

\newtheorem{problem}[theorem]{Problem}
\newtheorem{definition}[theorem]{Definition}
\newtheorem{example}[theorem]{Example}
\newtheorem{remark}[theorem]{Remark}
\numberwithin{equation}{subsection}
\newtheorem*{ack}{Acknowledgement}

\newcommand{\Env}{\operatorname{Env}}

\newcommand{\Alex}{\operatorname{Alex}}
\newcommand{\Aut}{\operatorname{Aut}}

\newcommand{\Conj}{\operatorname{Conj}}

\newcommand{\C}{\operatorname{C}}
\newcommand{\Inn}{\operatorname{Inn}}

\newcommand{\Hom}{\operatorname{Hom}}

\newcommand{\Z}{\operatorname{Z}}
\newcommand{\lcm}{\operatorname{lcm}}

\newcommand{\Fix}{\mathrm{Fix}}

\begin{document}
	\title{Fundamental $n$-quandles of links are residually finite}
	\author{Neeraj Kumar Dhanwani}	
	\author{Deepanshi Saraf}
	\author{Mahender Singh}
	
	\address{Department of Mathematical Sciences, Indian Institute of Science Education and Research (IISER) Mohali, Sector 81,  S. A. S. Nagar, P. O. Manauli, Punjab 140306, India.}
	\email{neerajk.dhanwani@gmail.com}
	\email{saraf.deepanshi@gmail.com}
	\email{mahender@iisermohali.ac.in}


\begin{abstract}
In this paper, we investigate the residual finiteness and subquandle separability of quandles, properties that respectively imply the solvability of the word problem and the generalized word problem for quandles. From Winker's work, we know that fundamental $n$-quandles of oriented links, which are canonical quotients of their fundamental quandles, are closely associated with $n$-fold cyclic branched covers of the 3-sphere branched over these links. We prove that the fundamental $n$-quandle of any oriented link in the 3-sphere is residually finite for each $n\ge 2$. 
This supplements the recent result by Bardakov, Singh and the third author on residual finiteness of fundamental quandles of oriented links, and the classification by Hoste and Shanahan of links whose fundamental $n$-quandles are finite for some $n$.  We also establish several general results on these finiteness properties and identify many families of quandles admitting them.
\end{abstract}
	
\maketitle
	
\section{Introduction}
Subgroup separability, as defined by Mal'cev in \cite{Malvev1958}, has applications in combinatorial group theory and low dimensional topology. If a finitely presented group is residually finite, then it has the solvable word problem \cite{MR0008991}. More generally, if a finitely presented group is subgroup separable, then it has the solvable generalised word problem \cite{Malvev1958}. In 3-manifold topology, subgroup separability has been applied to resolve immersion-to-embedding problems. For example, it is known that subgroup separability enables the passage from immersed incompressible surfaces to embedded incompressible surfaces in finite covers \cite{MR1851085}. The aim of this paper is to investigate these properties in the category of quandles with a focus on fundamental $n$-quandles of oriented links in the 3-sphere. 
\par 
Quandles are right distributive algebraic structures that appear as almost complete invariants of knots, and as non-degenerate set-theoretical solutions to the Yang-Baxter equation. More formally, a quandle is a set with a binary operation that satisfies axioms modelled on the three Reidemeister moves of planar diagrams of links in the 3-sphere. Joyce \cite{Joyce1979}  and Matveev \cite{MR0672410} independently proved that one can associate a quandle $Q(L)$ to each oriented link $L$, called the fundamental quandle of $L$,  which is an invariant of the isotopy type of $L$. Further, they showed that if $K_1$ and $K_2$ are two oriented knots with $Q(K_1) \cong Q(K_2)$, then there is a homeomorphism of the 3-sphere mapping $K_1$ onto $K_2$, not necessarily preserving the orientation of the ambient space. Although, the fundamental quandle is an almost complete invariant for oriented  knots, it is usually difficult to check whether two quandles are isomorphic. This has sparked the search for properties and invariants of these structures that are simpler to determine or calculate. Since fundamental quandles of links in the 3-sphere are always infinite, except for the case when the link in question is the unknot or the Hopf link \cite{MR3938088}, it is natural to ask whether these quandles exhibit more general finiteness properties. In this direction, Bardakov, Singh and the third author proved in \cite{MR3981139, MR4075375} that, along with many classes of quandles arising from groups, the fundamental quandles of oriented links in the 3-sphere are residually finite. Fox's classical work on link groups, where he examined their finite quotients and coloring invariants using finitely many colors, can be reinterpreted as the study of homomorphisms from the fundamental quandles of links onto finite quandles  \cite{MR0146828}. It follows from the residual finiteness of fundamental quandles of links that every oriented link, except the unknot, has a non-trivial coloring by a finite quandle.  Further, it is proved in  \cite{MR3981139}  that every finitely presented residually finite quandle has the solvable word problem.  However, in \cite{MR3440592}, Belk and McGrail showed that there exists a finitely presented quandle with an undecidable word problem. Based on the preceding result, it follows that such a quandle cannot be residually finite. 

In this paper, we carry out this study further in two directions. Firstly, we consider the residual finiteness of canonical quotients of fundamental quandles of oriented links in  the 3-sphere, called fundamental $n$-quandles, where $n \ge 2$ is an integer. 
The fundamental $n$-quandles were first considered by Joyce in \cite{Joyce1979}. Though fundamental $n$-quandles are less sensitive invariants of  oriented links, they are considerably more tractable and admit deep connections with $n$-fold cyclic branched covers of the 3-sphere branched over links, as developed by Winker in \cite{MR2634013}. Building on his work, Fish and Lisitsa \cite{Fish-Lisitsa} developed an algorithm that uses fundamental 2-quandles to detect the unknot efficiently, and conjectured that the fundamental 2-quandle of a knot is residually finite. 
\par
Further, Hoste and Shanahan in \cite{MR3704243} proved that the fundamental $n$-quandle of an  oriented  link $L$ in the 3-sphere is finite if and only if the fundamental group $\pi_{1}(\widetilde{M}_{n}(L))$ of the $n$-fold cyclic branched cover $\widetilde{M}_{n}(L)$ of the 3-sphere branched over $L$ is finite. As they state in their paper, Przytycki had communicated this to them as a conjecture. In fact, using Thurston's geometrisation theorem and Dunbar's classification of spherical 3-orbifolds, Hoste and Shanahan derived the complete list of oriented links which have a finite fundamental $n$-quandle for some $n \ge 2$. It turns out that most links have infinite fundamental $n$-quandles for nearly all values of $n$. Therefore, it is intriguing to ask whether these $n$-quandles are residually finite. An affirmative answer to this question would imply the solvability of the word problem for these non-associative algebraic structures, and would prove the conjecture of Fish and Lisitsa.
\par
Employing consequences of Thurston's geometrisation theorem and related results, we first prove that if $L$ is an oriented link in the 3-sphere,  then $\pi_{1}(\widetilde{M}_{n}(L))$ is abelian subgroup separable (Theorem \ref{piM abelian subgroup separable}). Using this result and a description of $\pi_{1}(\widetilde{M}_{n}(L))$ as a subgroup of a canonical quotient of the link group of $L$, we prove that the fundamental $n$-quandle of any oriented link is residually finite for each $n \ge 2$ (Theorem \ref{n quandles link res finite}).  
\par 
Secondly, we develop a general theory of subquandle separability of quandles, which implies the solvability of the generalised word problem for these algebraic structures. Among other results, we prove that certain subquandles of quandles arising from subgroup separable groups are separable (Proposition \ref{alex-subseparable}). We also establish subquandle separability of certain twisted unions of subquandle separable quandles (Proposition \ref{twisted union subquandle sep}), abelian quandles generated by two elements (Theorem \ref{two generated abelian quandle}), and finitely generated free abelian quandles (Theorem \ref{free abelian subquandle sep}).
\medskip

\section{Preliminaries}
This section reviews the essential preliminary material that will be employed throughout the paper. To set our convention, recall that, a {\it quandle} is a set $X$  with a binary operation $\ast$ that satisfies the following axioms:
\begin{enumerate}
\item $x \ast x=x$ for all $x \in X$.
\item Given $x, y \in X$, there exists a unique $z \in X$ such that $x=z*y$.
\item $(x \ast y) \ast z=(x \ast z) \ast(y \ast z)$ for all $x, y, z \in X.$
\end{enumerate}

The second quandle axiom is equivalent to saying that there exists a dual binary operation $(x , y )  \mapsto x*^{-1} y $ on $X$ such that $x * y = z$ if and only if $x = z *^{-1} y$ for all $x, y, z \in X$. Analogous to groups, quandles can be represented by their presentations.

\begin{example}\label{example fundamental quandle}
Let $L$ be an oriented link in the 3-sphere $\mathbb{S}^3$. In \cite[Section 4.5]{Joyce1979} and \cite[Section 6]{MR0672410}, Joyce and Matveev independently gave a topological construction of the  {\it fundamental quandle} $Q(L)$ of $L$,  and proved it to be an invariant of the isotopy type of  $L$. Further, they  proved that $Q(L)$ can also be obtained from a regular diagram $D$ of $L$. Suppose that $D$ has $s$ arcs and $t$ crossings. We assign labels $x_{1}, x_{2}, \ldots, x_{s}$ to the arcs of $D$, and then introduce the relation $r_l$ given by  $x_k*x_j=x_i$ or $x_k*^{-1}x_j=x_i$ at the $l$-th crossing of $D$ as shown in Figure \ref{fundamental quandle}.  It is known due to \cite{Joyce1979, MR0672410} that $$Q(L) \cong \langle x_1, x_2, \ldots, x_s \mid r_1, r_2, \dots, r_t\rangle.$$

\begin{figure}[hbt!]
	\begin{subfigure}{0.4\textwidth}
		\centering
		\begin{tikzpicture}[scale=0.6]
			\node at (0.6,-1.2) {{\small $x_k$}};
			\node at (-1.2,0.4) {{\small $x_j$}};
			\node at (2,1.2) {{\small $x_k*x_j=x_i$}};
			\begin{knot}[clip width=6, clip radius=4pt]
				\strand[->] (-2,0)--(2,0);
				\strand[->] (0,-2)--(0,2);
			\end{knot}
		\end{tikzpicture}
		\caption{Positive crossing}
	\end{subfigure}
	\begin{subfigure}{0.4\textwidth}
		\centering
		\begin{tikzpicture}[scale=0.6]
			\node at (2.3,1.2) {{\small $x_k*^{-1}x_j=x_i$}};
			\node at (-1.2,0.4) {{\small $x_j$}};
			\node at (0.6,-1.2) {{\small $x_k$}};
			\begin{knot}[clip width=6, clip radius=4pt]
				\strand[->] (2,0)--(-2,0);
				\strand[->] (0,-2)--(0,2);
			\end{knot}
		\end{tikzpicture}
		\caption{Negative crossing}
	\end{subfigure}
\caption{Quandle relations at crossings.}
	\label{fundamental quandle}
\end{figure}
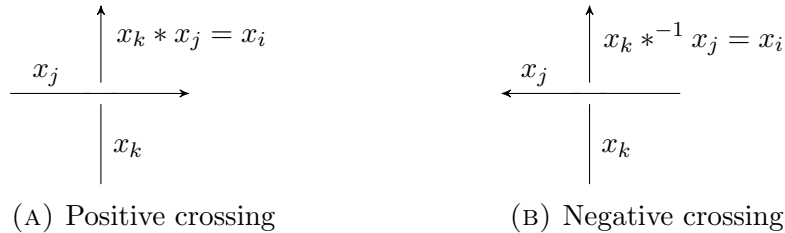
\end{example}

\begin{example}\label{more example quandle}
Though links in the 3-sphere are rich sources of quandles, many interesting examples arise particularly from groups, some of them will be used in later sections.
\begin{enumerate}
\item If $G$ is a group, then the set $G$ equipped with the binary operation $x*y= y^{-1} x y$ gives a quandle structure on $G$, called the {\it conjugation quandle}, and denoted by $\Conj(G)$.
\item Let $H$ be a subgroup of a group $G$ and $\alpha \in \Aut(G)$ that acts trivially on $H$. Then, the set $G/H$ of right cosets becomes a quandle with the binary operation $$Hx \ast Hy= H \alpha(xy^{-1})y.$$  In particular, if $\alpha$ is the inner automorphism of $G$ induced by an element $x_0$ in the centraliser of $H$ in $G$, then the quandle operation on $G/H$ becomes
\begin{align*}
	Hx \ast Hy= H x_{0}^{-1}xy^{-1}x_{0}y,
\end{align*}
and we denote this quandle by $(G/H,x_{0})$.
\item The preceding example can be  extended as follows. Let $G$ be a group, $\{x_i \mid i \in I \}$ be a set of elements of $G$, and $\{H_i \mid i \in I \}$ a set of subgroups of $G$ such that $H_i \le \C_G(x_i)$ for each $i$. Then, we can define a quandle structure on the disjoint union $\sqcup_{i \in I} G/H_i$ by
$$H_ix\ast H_jy=H_ix_i^{-1}xy^{-1}x_jy,$$ 
and denote this quandle by $\sqcup_{i \in I} (G/H_i,x_i ).$
\end{enumerate}
\end{example}
\par

If $X$ is a quandle and $x \in X$, then the map $S_x: X \rightarrow X$ given by $S_x(y)= y \ast x$ is an automorphism of $X$ fixing $x$. The group $\Inn(X)$ generated by such automorphisms is called the {\it inner automorphism group} of $X$. The orbits of $X$ under the natural action of  $\Inn(X)$ are called the {\it connected components} of $X$. Further, we say that $X$ is {\it connected} if $\Inn(X)$ acts transitively on $X$. 
\par 

Using the defining axioms  \cite[Lemma 4.4.7]{MR2634013}, any element of a quandle $X$ can be written in a left-associated product of the form
\begin{equation*}
\left(\left(\cdots\left(\left(x_0*^{\epsilon_1}x_1\right)*^{\epsilon_2}x_2\right)*^{\epsilon_3}\cdots\right)*^{\epsilon_{n-1}}x_{n-1}\right)*^{\epsilon_n}x_n,
\end{equation*}
where $x_i \in X$ and  $\epsilon_i \in \{ 1, -1 \}$. For simplicity, we write the preceding expression as
\begin{equation*}
x_0*^{\epsilon_1}x_1*^{\epsilon_2}\cdots*^{\epsilon_n}x_n.
\end{equation*}

Let  $n \ge2$ be an integer. A quandle $X$ is called an {\it $n$-quandle} if each $S_x$ has order dividing $n$. In other words, $X$ is an $n$-quandle if 
$$x \ast^{n} y:=x*\underbrace{y *y*\cdots *y}_{n ~\mathrm{times}}=x$$ for all $x,y \in X$.
\par

Given a quandle $X$ and an integer $n \ge2$, the $n$-quandle $X_n$ of $X$ is defined as the quotient of $X$ by the relations $$x*^n y:=x*\underbrace{y*y*\cdots*y}_{n~ \mathrm{times}}=x$$ for all $x,y \in X$.
\medskip


\subsection{Enveloping group}

To each quandle $X$, we associate its {\it enveloping group} $\Env(X)$, which is given by the presentation
\begin{gather}
	\Env(X)=\langle e_{x}, ~x \in X \mid e_{x \ast y}=e_{y}^{-1} e_{x} e_{y}~\textrm{for all}~x, y \in X \rangle.
\end{gather}

The association $X \mapsto \Env(X)$ defines a functor from the category of quandles to that of groups, which is left adjoint to the functor $G \mapsto \Conj(G)$ from the category of groups to that of quandles. 
\par

Analogously, there is a functor from the category of groups to the category of $n$-quandles for each $n\ge 2$. To be precise, given a group $G$, we consider the set
$$
Q_{n}(G)=\left\{x \in G \mid x^{n}=1\right\}
$$
equipped with the binary operation of conjugation, which is clearly an $n$-quandle. In the reverse direction, given an $n$-quandle $X$, we define its {\it $n$-enveloping group} to be
$$\Env_n(X)=\langle \bold{e}_x,~x \in X \mid \bold{e}_x^{n}=1,~ \bold{e}_{x  \ast y}=\bold{e}_y^{-1} \bold{e}_x \bold{e}_y~\textrm{for all}~x, y \in X \rangle.$$

It follows from \cite[Theorem 5.1.7]{MR2634013} that if a quandle $X$ has the presentation
$$
X=\left\langle x_{1}, x_{2}, \ldots, x_{s} \mid r_{1}, r_{2}, \ldots, r_{t}\right\rangle,
$$
then $\Env(X)$ has the presentation
$$
\Env(X)=\left\langle e_{{x}_{1}},e_{{x}_{2}},\ldots,  e_{{x}_{s}} \mid \bar{r}_{1}, \bar{r}_{2}, \ldots, \bar{r}_{t}\right\rangle,
$$
where each relation $\bar{r}_i$ is obtained from the relation $r_i$ by replacing each expression $x*y$ by $e_y^{-1} e_x e_y$ and $x*^{-1}y$ by $e_y e_x e_y^{-1}$. Furthermore, if $X$ is an $n$-quandle, then it follows that $\Env_{n}(X)$  has the presentation
$$
\Env_n(X)=\left\langle \bold{e}_{x_1}, \bold{e}_{x_2}, \ldots,\bold{e}_{x_s} \mid \bold{e}_{x_1}^n=1, \bold{e}_{x_2}^n=1, \ldots,  \bold{e}_{x_s}^n=1, \bold{r}_{1}, \bold{r}_{2}, \ldots, \bold{r}_{t}\right\rangle,
$$
where each relation $\bold{r}_i$ is obtained from the relation $r_i$ by replacing each expression $x*y$ by $\bold{e}_y^{-1} \bold{e}_x \bold{e}_y$ and $x*^{-1}y$ by $\bold{e}_y \bold{e}_x \bold{e}_y^{-1}$. Observe that $$\Env_n(X) \cong \Env(X)/ \llangle e_y^n, ~ y \in X\rrangle.$$ 
	

\subsection{Homogeneous representation of $n$-quandles}
Given an $n$-quandle $X$, there is a right action of $\Env(X)$ on $X$, which on generators of  $\Env(X)$ is given by $$x \cdot e_y  = x*y$$ for $x, y \in X$. Let $g=e_{y_1}^{\epsilon_1}e_{y_2}^{\epsilon_2} \cdots e_{y_r}^{\epsilon_r}$ be an element of $\Env(X)$, where  $y_i \in X$ and $\epsilon_i \in \{1, -1 \} $. Since $X$ is an $n$-quandle, for each $x \in X$, we have
\begin{eqnarray*}
x \cdot (g e_y^n g^{-1})  &=& x*^{\epsilon_1} y_1 *^{\epsilon_2} y_2* \cdots *^{\epsilon_r} y_r *\underbrace{y * y* \cdots *y}_{n-\textrm{times}} *^{-\epsilon_r} y_r* \cdots *^{-\epsilon_2} y_2 *^{-\epsilon_1} y_1\\
&=& x*^{\epsilon_1} y_1 *^{\epsilon_2} y_2* \cdots *^{\epsilon_r} y_r *^{-\epsilon_r} y_r* \cdots *^{-\epsilon_2} y_2 *^{-\epsilon_1} y_1\\
&=& x,
\end{eqnarray*}
and hence the action descends to an action of $\Env_n(X)$ on $X$. The following result can be proved easily, and we present a proof for the benefit of the reader.

\begin{prop}\label{finiteness of n-quandle}
Let $X$ be an $n$-quandle and $\{x_i \mid i \in I\}$ a set of representatives of orbits of $X$ under the action of $\Env(X)$. Let $H_i$ be the stabiliser of $x_{i}$ in $\Env_n(X)$ under the above action. Then $H_i$ lies in the centraliser of $\bold{e}_{x_i}$ in $\Env_n(X)$ and the orbit map induces an isomorphism $\sqcup_{i \in I} (\Env_n(X)/H_i,\bold{e}_{x_i}) \cong X$ of quandles.
\end{prop}

\begin{proof}
	Let $h \in H_i$ such that $h= \bold{e}_{x_1}^{\epsilon_1} \bold{e}_{x_2}^{\epsilon_2} \cdots \bold{e}_{x_r}^{\epsilon_r}$ for some $x_j \in X$ and $\epsilon_j \in \{1, -1 \}$. Then, we see that
	\begin{eqnarray*}
		h^{-1} \bold{e}_{x_i} h &=&  \bold{e}_{x_r}^{-\epsilon_r} \bold{e}_{x_{r-1}}^{-\epsilon_{r-1}} \cdots \bold{e}_{x_1}^{-\epsilon_1}\bold{e}_{x_i}  \bold{e}_{x_1}^{\epsilon_1} \bold{e}_{x_2}^{\epsilon_2} \cdots \bold{e}_{x_r}^{\epsilon_r} \\
		&=&  \bold{e}_{x_i *^{\epsilon_1} x_1 *^{\epsilon_{2}} x_{2}* \cdots *^{\epsilon_r} x_r}\\ 
		&=&  \bold{e}_{x_i \cdot h}\\  
		&=&  \bold{e}_{x_i},
	\end{eqnarray*}
	and hence  $H_i$ lies in the centraliser of $\bold{e}_{x_i}$ in $\Env_n(X)$. Thus, we obtain the quandle $\sqcup_{i \in I} (\Env_n(X)/H_i,\bold{e}_{x_i})$. Since $\Env(X)$ acts transitively on connected components of $X$, the induced action of $\Env_n(X)$ is also transitive on connected components of  $X$. Further, since $H_i$ is the stabiliser of $x_{i}$ in $\Env_n(X)$, we have a bijection $$\phi:  \sqcup_{i \in I} (\Env_n(X)/ H_i,\bold{e}_{x_i} ) \rightarrow  X$$ induced by the orbit map $H_ig \mapsto x_i \cdot g$. It remains to check that $\phi$ is a quandle homomorphism. Indeed, for $u, v \in \Env_n(X)$, we have
	\begin{align*}
		\phi(H_iu \ast H_jv) & = \phi(H_i\bold{e}_{x_i}^{-1} u v^{-1} \bold{e}_{x_{j}} v)\\
		& =  x_i\cdot (\bold{e}_{x_i}^{-1} u v^{-1} \bold{e}_{x_{j}} v)\\
		& =  x_i\cdot (u v^{-1} \bold{e}_{x_{j}} v),~ \textrm{since $x_i\cdot \bold{e}_{x_i}^{-1}= x_i *^{-1} x_i=x_i$}\\
		& =  (((x_i\cdot u) \cdot v^{-1} )*x_j ) \cdot v\\
		& =  (x_i\cdot u) *(x_j  \cdot v)\\
		& = \phi(H_iu)\ast \phi(H_jv),
	\end{align*}
	which completes the proof.
\end{proof}	

As a consequence, we generalise a result of Hoste and Shanahan \cite[Theorem 3.2]{MR3704243} to arbitrary $n$-quandles. See also \cite[Proposition 3.1]{MR4447657} and \cite[Section 3.5]{Joyce1979} for the one way implication.

\begin{corollary}\label{quandle iff env finite}
Let $X$ be an $n$-quandle for some $n \ge 2$. Then $X$ is  finite  if and only if $\Env_n(X)$ is finite.
\end{corollary}
\begin{proof}
By \cite[Theorem 3.2]{MR3704243}, if $X$ is finite, then $\Env_n(X)$ is finite. The converse follows from Proposition \ref{finiteness of n-quandle}.
\end{proof}

\medskip


\section{Residual finiteness and fundamental $n$-quandles of links}\label{section 3}

\subsection{Residual finiteness and subquandle separability of quandles}
We begin by recalling the definition of a subgroup separable group.

\begin{definition}
A subset $S$ of a group $G$ is said to be {\it separable} in $G$ if for each $x \in G\setminus S$, there exists a finite group $F$ and a  group homomorphism $\phi: G \to F$ such that $\phi(x) \not\in \phi(S)$. If the singleton set $S$ consisting of only the identity element is separable, then $G$ is called {\it residually finite}. If each finitely generated subgroup of $G$ is separable, then $G$ is called {\it subgroup separable}. 
\end{definition}

Recall that, the {\it profinite topology} on a group $G$ has a basis consisting of right cosets of all finite index subgroups of $G$. By definition, every right coset of a finite index subgroup is closed in the profinite topology. An easy check shows that a subgroup $H$ of $G$ is separable in $G$ if and only if $H$ is closed in the profinite topology on $G$ \cite{MR2414358,MR1398726}. We will use this equivalent definition of subgroup separability to prove the following result.

\begin{prop}\label{finite index separable}
Let $H$ and $K$ be subgroups of $G$ such that $[G:K]$ is finite. Let $L=H \cap K$ such that $[H:L]$ is finite and $L$ is a separable subgroup of $K$. Then $H$ is a separable subgroup of $G$.
\end{prop}
\begin{proof}
Since $L$ is a separable subgroup of $K$,  it is closed in the profinite topology on $K$. Since $[G:K]$ is finite, it follows that $K$ is closed in the profinite topology on $G$. Consequently, $L$ is closed in the profinite topology on $G$. Since $[H:L]$ is finite,  $H$ is a finite union of right cosets of $L$, and hence it is closed in the profinite topology on $G$. Thus, it follows that $H$ is a separable subgroup of $G$.
\end{proof}

As a consequence of Proposition \ref{finite index separable}, we recover the following well-known result.

\begin{corollary}\label{finite index res finite}
Let $G$ be a group admitting a residually finite subgroup of finite index. Then $G$ is residually finite.
\end{corollary}

In analogy with groups, we introduce the following definition for quandles.

\begin{definition}
A subset $S$ of a quandle $X$ is said to be {\it separable} in $X$  if for each $x \in X\setminus S$, there exists a finite quandle $F$ and a quandle homomorphism $\phi : X \rightarrow F$ such that $\phi(x) \notin \phi(S)$. If each singleton set is separable, then $X$ is called {\it residually finite}. If each finitely generated subquandle of $X$ is separable, then $X$ is called  {\it subquandle separable}.
\end{definition}

We note that residual finiteness of fundamental quandles of oriented links has been established recently in \cite{MR3981139, MR4075375}. 
\par

Let $G$ be a group and $H$ a finitely generated subgroup of $G$. The {\it generalised word problem} is the problem of deciding for an arbitrary element $w$ in $G$ whether or not $w$ lies in $H$. Let $X$ be a quandle and $Y$ its  finitely generated subquandle. We can define the generalised word problem for quandles as the problem of  deciding for an arbitrary element $w$ in $X$ whether or not $w$ lies in $Y$. The following is an analogue of the corresponding result for groups.

\begin{prop}
A finitely presented subquandle separable  quandle has the solvable generalised word problem.
\end{prop}

\begin{proof}
Let $X=\langle S \mid R \rangle$ be a finitely presented subquandle separable quandle and $Y$ be its finitely generated subquandle. Let $x$ be an element of $X$. We describe two procedures to determine whether $x$ is in $Y$ or not. The first procedure lists all the elements obtained from the generators of $Y$ using the relations in $R$ (and quandle axioms). If, at some stage,  $x$ turns up as one of these elements, then $x \in Y$.
\par
The second procedure lists all the finite quandles. Since $X$ is finitely generated, for each finite quandle $F$, the set $\Hom(X, F)$ of all quandle homomorphisms is finite. Now, for each homomorphism $\phi \in \Hom(X, F)$, we look for $\phi(x)$ and $\phi(Y)$ in $F$ and check whether or not $\phi(x) \in \phi(Y)$. If, at some stage, $\phi(x)\notin \phi(Y)$, then $x\notin Y$. Since $X$ is a finitely presented subquandle separable quandle and $Y$ is a finitely generated subquandle of $X$, one of the above procedures must stop in finite time.
\end{proof}

The following result from \cite[Proposition 3.4]{MR4075375} will be used later in proving our main result on  fundamental $n$-quandles of oriented links.

\begin{prop}\label{res-finite-quandle}
Let $G$ be a group, $\{x_i \mid i \in I \}$ be a finite set of elements of $G$, and $\{H_i \mid i \in I \}$ a finite set of subgroups of $G$ such that $H_i \le \C_G(x_i)$ for each $i$. If each $H_i$ is separable in $G$, then the quandle $\sqcup_{i \in I}(G/H_i,x_i)$ is residually finite.
\end{prop}
\medskip

\subsection{Residual finiteness of fundamental $n$-quandles of oriented links}

Let $L$ be an oriented link in $\mathbb{S}^3$ with components $K_1, K_2, \ldots, K_m$. Then, as in Example \ref{example fundamental quandle}, we can associate the {\it fundamental quandle} $Q(L)$ to the link $L$, which is constructed from a regular diagram $D$ of $L$ and admits the presentation
$$
Q(L)=\left\langle x_{1}, x_2, \ldots, x_{s} \mid r_{1}, r_2, \ldots, r_{t}\right\rangle,
$$
where each $x_i$ is an arc of $D$, and each relation $r_l$ is given by  $x_k*x_j=x_i$ or $x_k*^{-1}x_j=x_i$ as per the corresponding crossing in $D$. The preceding presentation yields the presentation of $\Env(Q(L))$, which is precisely the Wirtinger presentation of the link group $\pi_1(\mathbb{S}^{3} \setminus L)$, that is, $\Env(Q(L)) \cong \pi_1(\mathbb{S}^{3} \setminus L)$.
\par

A less sensitive, but presumably more tractable invariant of an oriented link $L$ is the {\it fundamental $n$-quandle} $Q_n(L)$ defined for each natural number $n \ge 2$ as the quandle with the presentation
$$Q_{n}(L)=\langle x_{1}, x_2, \ldots, x_{s} \mid r_{1}, r_2, \ldots, r_{t}, u_{1}, u_2, \ldots, u_{k}\rangle,$$
where each relation $u_{\ell}$ is of the form $x_{i} \ast^{n} x_{j}=x_i$ for distinct generators $x_{i}$ and $x_{j}$. It follows from \cite[Proposition 3.1]{DhanwaniSingh2024} that the additional relations $u_{1}, u_2, \ldots, u_{k}$ suffice to make $Q_{n}(L)$ an $n$-quandle. If $L$ is a link with more than one component, then both $Q(L)$ and $Q_n(L)$ are disconnected with one component $Q^i(L)$ and $Q^i_n(L)$, respectively, for each component $K_i$ of $L$.
\par 

Passing from the presentation of $Q_{n}(L)$ to the presentation for $\Env_{n}(Q_{n}(L))$, we see that $\Env_{n}\left(Q_{n}(L)\right)$ is a quotient of $\Env(Q(L))$. In fact, we may present $\Env_{n}\left(Q_{n}(L)\right)$ by adjoining the relations $x^{n}=1$ for each Wirtinger generator $x$ of the link group $\pi_{1}\left(\mathbb{S}^{3} \setminus L\right)\cong \Env(Q(L))$. While the fundamental quandle of a non-trivial knot, except the Hopf link, is always infinite, its corresponding fundamental $n$-quandle can be finite. In fact, it is known due to Hoste and Shanahan \cite[Theorem 3.1]{MR3704243} that, if $\widetilde{M}_{n}(L)$ is the $n$-fold cyclic branched cover of $\mathbb{S}^{3}$, branched over an  oriented  link $L$, then $Q_{n}(L)$ is finite if and only if $\pi_{1}(\widetilde{M}_{n}(L))$ is finite.
\par

\begin{prop}
If $L$ is an oriented link with any finite number of components and $n \ge 2$, then $\Env_n\left(Q_{n}(L)\right)$ is a residually finite group.
\end{prop}

\begin{proof}
In view of \cite[Theorem 5.2.2]{MR2634013}, for each $n \ge 2$, we have $$\pi_{1}\left( \widetilde{M}_{n}(L)\right) \cong 
E_n^0,$$ where $E_n^0$ is the subgroup of $\Env_n(Q_{n}(L))$ consisting of all elements whose total exponent sum equals to zero modulo $n$. Since fundamental groups of 3-manifolds are residually finite \cite{hempel}, it follows that $\pi_{1}(\widetilde{M}_{n}(L))$, and hence $E_n^0$ is residually finite. By  \cite[Section 3]{MR3704243}, the subgroup $E_n^0$ is of finite index in $\Env_n\left(Q_{n}(L)\right)$. Hence, by Corollary \ref{finite index res finite}, $\Env_n\left(Q_{n}(L)\right)$ is residually finite.
\end{proof}

Focusing specifically on knots, we can actually prove that the enveloping groups of their fundamental $n$-quandles are residually finite.

\begin{prop}
If $K$ is an oriented knot and $n \ge 2$, then $\Env\left(Q_{n}(K)\right) \cong \pi_{1}(\widetilde{M}_{n}(K))\rtimes \mathbb{Z}$. Moreover, $\Env\left(Q_{n}(K)\right)$ is a residually finite group. 
\end{prop}

\begin{proof}
Since $K$ is a knot, the fundamental $n$-quandle $Q_n(K)$ is connected and we have $E_n^0 =[\Env_n(Q_{n}(K)),\Env_n(Q_{n}(K))]$. By \cite[Corollary 4.2]{MR4447657}, we have
	$$
	\Env(Q_n(K)) \cong [\Env_n(Q_{n}(K)),\Env_n(Q_{n}(K))] \rtimes \mathbb{Z}.
	$$
Further, by \cite[Remark 5.1.5, Theorem 5.2.2]{MR2634013}, we have
	$$
	[\Env_n(Q_{n}(K)),\Env_n(Q_{n}(K))] \cong \pi_{1}(\widetilde{M}_{n}(K))
	$$
and hence $\Env(Q_n(K)) \cong \pi_{1}(\widetilde{M}_{n}(K)) \rtimes \mathbb{Z}$. For the second assertion, recall from \cite[Theorem 7]{MR0310044} that  semi-direct products of finitely generated residually finite groups are residually finite. Since both $\pi_{1}(\widetilde{M}_{n}(K))$ and $\mathbb{Z}$ are residually finite,  it follows that $\Env\left(Q_{n}(K)\right)$ is residually finite. 
\end{proof}

\begin{theorem}\label{piM abelian subgroup separable}
If $L$ is an oriented link and $n \ge 2$, then $\pi_{1}(\widetilde{M}_{n}(L))$ is abelian subgroup separable.
\end{theorem}

\begin{proof}
By the Prime Decomposition Theorem \cite[Theorem 1]{MR0142125}, every compact connected orientable 3-manifold without boundary which is not the 3-sphere is homeomorphic to a connected sum of prime 3-manifolds. Thus, we have
$\widetilde{M}_{n}(L)= N_1 \# N_2 \# \cdots \# N_q$, where each $N_i$ is a compact connected orientable prime 3-manifold without boundary. The van-Kampen Theorem gives $$\pi_{1}(\widetilde{M}_{n}(L)) \cong \pi_1(N_1) * \pi_1(N_2) * \cdots * \pi_1(N_q).$$ Recall that a 3-manifold is irreducible if every embedded 2-sphere bounds a 3-ball. It is known that, with the exception of 3-manifolds $\mathbb{S}^{3}$ and $\mathbb{S}^{1} \times \mathbb{S}^{2}$, an orientable manifold is prime if and only if it is irreducible \cite[Lemma 1]{MR0142125}. By \cite[Proposition 6]{MR1851085}, a free product of abelian subgroup separable groups is abelian subgroup separable. Thus, it suffices to prove that  $\pi_1(N)$ is abelian subgroup separable for each compact connected irreducible orientable 3-manifold $N$ without boundary. 
\par

By the Geometrisation Theorem \cite[Theorem 1.7.6]{MR3444187}, if $N$ is such a  3-manifold, then there exists a (possibly empty)
collection of disjointly embedded incompressible tori $T_1,\ldots,T_p$ in $N$ such that each component of $N$ cut along $T_1\cup \cdots \cup T_p$ is hyperbolic or Seifert fibered. 
\begin{itemize}
\item If $N$ is Seifert fibered, then by \cite[Corollary 5.1]{MR1154898}, $\pi_1(N)$ is double coset separable, and hence it is abelian subgroup separable.
\item If $N$ is hyperbolic, then $\pi_1(N)$ is subgroup separable by \cite[Corollary 4.2.3]{MR3444187}, and hence it is abelian subgroup separable.
\item If $N$ admits an incompressible torus, then it is Haken \cite[p.45, A.10]{MR3444187}. It follows from \cite[Theorem 1]{MR1851085} that $\pi_1(N)$ is abelian subgroup separable.
\end{itemize}
This completes the proof of the theorem.
\end{proof}

Let $L$ be an oriented link with components $K_1, K_2, \ldots, K_m$ and $n \ge 2$. For each $i$, let $m_i$ and $\ell_i$ be the fixed meridian and the longitude of the component $K_i$, respectively. By abuse of notation, we also denote by $m_i$ and $\ell_i$ their images in the quotient $\Env_n(Q_{n}(L))$.  In view of the isomorphism $\pi_{1}\left( \widetilde{M}_{n}(L)\right) \cong E_n^0$, we can further view each $\ell_i$ as an element of $\pi_{1}\left( \widetilde{M}_{n}(L)\right)$.

Proposition \ref{piM abelian subgroup separable} leads to the following result.

\begin{corollary}\label{l_i is separable}
Let $L$ be an oriented link with components $K_1, K_2, \ldots, K_m$ and $n \ge 2$. Then, the subgroup $\langle \ell_i \rangle$ generated by the fixed longitude $\ell_i$ of $K_i$ is subgroup separable in $\pi_{1}(\widetilde{M}_{n}(L))$ for each $i$.
\end{corollary}

\begin{corollary}\label{m_k,l_k separable}
Let $L$ be an oriented link with components $K_1, K_2, \ldots, K_m$ and $n \ge 2$. For each $i$, let $m_i$ and $\ell_i$ be the fixed meridian and the longitude of $K_i$, respectively.  Then,  $P_i=\langle m_i, \ell_i \rangle$ is a separable subgroup of $\Env_n(Q_n(L))$ for each $i$.
\end{corollary}

\begin{proof}
By  \cite[Section 3]{MR3704243}, the subgroup $\pi_{1}(\widetilde{M}_{n}(L))$ is of finite index in $\Env_n(Q_n(L))$. Since $\ell_i \in \pi_{1}\left( \widetilde{M}_{n}(L)\right) \cong E_n^0$, it follows that $P_i \cap \pi_{1}(\widetilde{M}_{n}(L))=\langle \ell_i \rangle$.  Also, we have $[P_i : \langle \ell_i \rangle ]=n$. Further, by Corollary \ref{l_i is separable}, $\langle \ell_i \rangle$ is a separable subgroup of $\pi_{1}\left( \widetilde{M}_{n}(L)\right)$. Hence, by Proposition \ref{finite index separable}, $P_i$ is a separable subgroup of $\Env_n(Q_n(L))$.
\end{proof}

We can now deduce the main result of this section.

\begin{theorem}\label{n quandles link res finite}
If $L$ is an oriented link and $n \ge 2$, then the fundamental $n$-quandle $Q_n(L)$ is residually finite.
\end{theorem}

\begin{proof}
Let $L$ be an oriented link with components $K_1, K_2, \ldots, K_m$. Let $m_i$ and $\ell_i$ be the fixed meridian and the longitude of $K_i$, respectively. Then, by \cite[Theorem 1.1]{MR3704243} or Proposition \ref{finiteness of n-quandle}, we can write 
$$Q_n(L) \cong \sqcup_{i=1}^m (\Env_n(Q_n(L))/P_i,m_i),$$ where $P_i=\langle m_i, \ell_i \rangle$. Corollary \ref{m_k,l_k separable} implies that each $P_i$ is a separable subgroup of $\Env_n(Q_n(L))$. The result now follows from Proposition \ref{res-finite-quandle}.
\end{proof}

By \cite[Theorem 5.11]{MR3981139}, every finitely presented residually finite quandle has the solvable word problem. Thus, the preceding theorem leads to the following corollary.

\begin{corollary}
If $L$ is an oriented link and $n \ge 2$, then the fundamental $n$-quandle $Q_n(L)$ has the solvable word problem.
\end{corollary}

We conclude this section with the following natural problem.

\begin{problem}
Classify links in $\mathbb{S}^3$ whose fundamental quandles and fundamental $n$-quandles (for $n \ge 2$) are subquandle separable. As expected, the problem is intimately related to subgroup separability of link groups.
\end{problem}
\medskip


\section{Residual finiteness of general quandles}
In this section, we establish residual finiteness of some classes of  quandles.

\begin{definition}
Let $S$ be a non-empty set and $n \ge 2$.  A quandle $FQ_n(S)$ containing $S$ is called a {\it free $n$-quandle} on the set $S$, if given any map $\phi:S\to X$, where $X$ is an $n$-quandle, there is a unique quandle homomorphism $\bar{\phi}:FQ_n(S)\to X$ such that $\bar{\phi}|_S= \phi$.
\end{definition}

Consider the free product $G=\ast_{|S|} \, \mathbb{Z}_n$ of cyclic groups of order $n$, one for each element of $S$. It is known from \cite[Section 2.11, Corollary 2]{Joyce1979} that $FQ_n(S)$ is a subquandle of $\Conj(G)$ consisting of conjugates of the generators of $G$.

\begin{prop} Let $S$ be a non-empty set and $n \ge 2$. Then $FQ_n(S)$ is a residually finite quandle.
\end{prop}

\begin{proof}
Let $G=\ast_{|S|}\,\mathbb{Z}_n$ be the free product of cyclic groups of order $n$, one for each element of $S$. Since free products of residually finite groups are residually finite, it follows that $G$ is a residually finite group, and hence $\Conj(G)$ is a residually finite quandle. Consequently, $FQ_n(S)$ being a subquandle of $\Conj(G)$ is also residually finite. 
\end{proof}

Let $G$ be a group and $\alpha \in \Aut(G)$. Then the {\it twisted conjugation quandle} $\Conj(G,\alpha)$ is the set $G$ equipped with the quandle operation $$x*y=\alpha(y^{-1} x)y.$$ These structures appeared in Andruskiewitsch-Gra\~{n}a \cite[Section 1.3.7]{MR1994219} as twisted homogeneous crossed sets. We prefer calling them twisted conjugation quandles since $\Conj(G,\alpha) = \Conj(G)$ when $\alpha$ is the identity map.

\begin{prop}
Let $G$ be a finitely generated residually finite group and $\alpha \in \Aut(G)$. Then the twisted conjugation quandle $\Conj(G,\alpha)$ is residually finite.
\end{prop}

\begin{proof}
By \cite{MR4564617}, there is an embedding of quandles $\Conj(G, \alpha)\hookrightarrow \Conj(G \rtimes_\alpha \mathbb{Z})$, where the action of $\mathbb{Z}$ on $G$ is defined via the automorphism $\alpha$. By \cite[Theorem 7, p.29]{MR0310044}, a split extension of a residually finite group by a  finitely generated residually finite group is again residually finite. Thus, $G \rtimes_{\alpha} \mathbb{Z}$ is a residually finite group, and hence $\Conj(G \rtimes_{\alpha} \mathbb{Z})$ is a residually finite quandle. 
Consequently, $\Conj(G,\alpha)$ being a subquandle of  $\Conj(G \rtimes_{\alpha} \mathbb{Z})$ is also residually finite. 
\end{proof}

A quandle $X$ is said to be {\it abelian} if $\Inn(X)$ is an abelian group. Equivalently, $X$ is abelian if $(x*y)*z=(x*z)*y$ for all $x, y, z \in X$.  In \cite{MR4116819}, a description of all finite quandles with abelian enveloping groups has been given, and it has been proved that any such quandle must be abelian.

\begin{prop} \label{Env(X) rf for abelian}
If $X$ is a finitely generated abelian quandle, then $\Env(X)$ is a residually finite group.
\end{prop}
\begin{proof}
We have the central extension
		$$1 \rightarrow \ker(\psi_X) \rightarrow \Env(X) \xrightarrow{\psi_X} \Inn(X) \rightarrow 1,$$
where $\psi_X(e_x)=S_x$ for each $x \in X$. Since $X$ is finitely generated, $\Env(X)$ is finitely generated. Further, since $X$ is abelian, $\Inn(X)$ is an abelian group. Thus, $\Env(X)$ is a finitely generated metabelian group. It follows from the well-known result of Hall \cite[Theorem 1]{MR0110750} that $\Env(X)$ is residually finite.
\end{proof}

\begin{prop}\label{X is rf for abelian}
A finitely generated abelian quandle is residually finite.
\end{prop}
\begin{proof}
Let $X$ be a finitely generated abelian quandle. Let $\{x_i \mid i \in I\}$ be a finite set of representatives of orbits of $X$ under the action of $\Inn(X)$, and let $H_i$ be the stabiliser of $x_i$ under this action. Since $H_i \le \C_{\Inn(X)}(S_{x_i})$, arguments in the proof of proposition in \cite[Section 2.4]{Joyce1979} shows that  $$X \cong \sqcup_{i \in I} (\Inn(X)/H_i,S_{x_i})$$ as quandles. Since $\Inn(X)$ is abelian, each of its subgroup, in particular, each $H_i$ is separable in $\Inn(X)$. Thus, by Proposition \ref{res-finite-quandle}, $X$ is a residually finite quandle.
\end{proof}
\medskip


\section{Subquandle separability of general quandles}
In this section, we explore subquandle separability of some classes of quandles.

\begin{prop}\label{finite-sub-separable}
If $X$ is a residually finite quandle, then every finite subquandle of $X$ is separable.
\end{prop}
\begin{proof}
Let $S=\{x_1, x_2, \ldots, x_k\}$ be a finite subquandle of $X$. For each $z \in X \setminus S$ and each $i \in \{1,2, \ldots, k\},$ there exists a finite quandle $Y_i$ and a quandle homomorphism $\phi_i:X \rightarrow Y_i$ such that $\phi_i(z) \neq \phi_i(x_i)$. Define $\Phi:X \rightarrow \prod_{i=1}^k Y_i$ by $\Phi(x)= (\phi_1(x), \phi_2(x), \ldots, \phi_k(x))$. Then, we have $\Phi(z) \notin \Phi(S)$, which is desired.
\end{proof}

By \cite[Proposition 3.2]{MR3981139}, a trivial quandle is residually finite. This together with the preceding proposition yields the following result.

\begin{corollary}
A trivial quandle is subquandle separable.
\end{corollary}

If $G$ is a group and $\alpha \in \Aut(G)$, then the  binary operation $$x*y=\alpha(xy^{-1})y$$ gives a quandle structure on $G$, denoted by $\Alex(G, \alpha)$. These quandles are called  {\it generalized Alexander quandles}. Note that, $\Alex(G, \alpha)$ is a special case of Example \ref{more example quandle}(2) when $H$ is the trivial subgroup. If $G$ is abelian, then $\Alex(G,\alpha)$ is precisely the twisted conjugation quandle $\Conj(G,\alpha)$. The following results generalise \cite[Proposition 4.1 and Proposition 4.2]{MR3981139}.

\begin{prop}\label{alex-subseparable}
Let $G$ be a subgroup separable group, $H$ a finitely generated subgroup of $G$ and $\alpha$ an inner automorphism of $G$ such that $\alpha(H) =H.$ Then the following assertions hold:
\begin{enumerate}
\item $\Alex(H,\alpha|_H)$ is a  separable subquandle of $\Alex(G,\alpha)$.
\item $\Conj(H,\alpha|_H)$ is a  separable subquandle of $\Conj(G,\alpha)$.
\end{enumerate}
\end{prop}

\begin{proof}
Let $\alpha$ be the inner automorphism induced by $g \in G$, and let $x \in \Alex(G, \alpha) \setminus \Alex(H, \alpha|_H)$.  By subgroup separability of $G$, there exists a finite group $F$ and a group homomorphism $\phi:G \rightarrow F$ such that $\phi(x) \notin \phi(H)$. Let  $\beta$ be the inner automorphism of $F$ induced by $\phi(g)$. It follows that $\phi$ viewed as a map $\Alex(G, \alpha) \rightarrow \Alex(F, \beta)$  is a quandle homomorphism with $\phi(x) \notin \phi( \Alex(H,\alpha|_H))$. Hence, $\Alex(H, \alpha|_H)$ is a separable subquandle of $\Alex(G,\alpha)$, which proves (1). The proof of assertion (2) is analogous.
\end{proof}

Let $w = w(x, y)$ be a word in the free group on two generators $x$ and  $y$. For each group $G$, the word $w$ defines a binary operation $*_w : G \times G \to G$ given by $g *_w h = w(g, h)$ for all $g, h \in G$. In \cite[Proposition 3.1]{MR4406425}, a complete characterisation of words $w$ for which $(G, *_w)$ is a quandle for each group $G$ has been given. In fact, the only possible words are $w(x, y) = y x^{-1} y$ or $w(x, y) = y^{-n} x y^n$ for some $n \in \mathbb{Z}$.

\begin{prop}\label{conj-subseparable-quandle}
Let $w(x, y) = y x^{-1} y$ or $w(x, y) = y^{-n} x y^n$ be a word in the free group on two generators. Let $G$ be a subgroup separable group. Then each finitely generated subquandle $(H, *_w)$ of $(G, *_w)$, where $H$ is a subgroup of $G$, is separable in $(G, *_w)$.
\end{prop}

\begin{proof}
Let $(H, *_w)$ be a finitely generated subquandle of $(G, *_w)$. It follows that $H$ is a finitely generated subgroup of $G$. Since $G$ is subgroup separable, for any $x \in G \setminus H$, there is a finite group $F$ and a group homomorphism $\phi:G \to F$ such that $\phi(x) \not\in \phi(H)$. Viewing $\phi$ as a map $(G, *_w) \to (F, *_w)$, we get $\phi(x) \not\in \phi((H, *_w))$. Hence, $(H, *_w)$  is a separable subquandle of $(G, *_w)$.
\end{proof}

\begin{prop}\label{intersectionp of subquandles}
Let $X$ be a quandle. Then the following assertions hold: 
\begin{enumerate}
\item If $\{X_i \mid i \in I \}$ is a family of separable subquandles of $X$, then $\cap_{i \in I}X_i$ is a separable subquandle of $X$.
\item If  $\{ \alpha_i\mid i \in I \}$ is a family of automorphisms of $X$, then $\cap_{i \in I} \Fix(\alpha_i)$ is a separable subquandle of $X$. Here, $\Fix(\alpha_i)= \{x \in X \mid \alpha_i(x)=x \}$ for each $i$.
\end{enumerate}
\end{prop}

\begin{proof}
If $x\in X\setminus \cap_{i \in I}X_i$, then there exists $j$ such that $x\in X\setminus X_j$. Since $X_j$ is separable in $X$, we have a surjective quandle homomorphism $\phi:X\to Y$, where $Y$ is finite, such that $\phi(x)\notin \phi(X_j)$. Thus, $\phi(x)\notin \phi(\cap_{i \in I}X_i)$, which is desired.
\par		
By  \cite[Proposition 6.4]{MR3981139}, each $\Fix(\alpha_i)$ is a separable subquandle of $X$. It now follows from assertion (1) that  $\cap_{i \in I} \Fix(\alpha_i)$  is a separable subquandle of $X$.
\end{proof}

Let $(X_1,\star_1)$ and $(X_2,\star_2)$ be quandles,  $f\in \C_{\Aut(X_1)}(\Inn(X_1))$ and $g\in \C_{\Aut(X_2)}(\Inn(X_2))$. Then, by \cite[Section 9]{MR3948284}, $X:=X_1\sqcup X_2$ turns into a quandle with the operation $\ast$ defined as 
$$x\ast y=\begin{cases}
		x\star_1 y \quad \text{ if } x,y\in X_1,\\
		x \star_2 y \quad \text{ if } x,y\in X_2,\\
		f(x)  \quad \text{ if } x\in X_1 \text{ and } y \in X_2,\\
		g(x) \quad  \text{ if } x\in X_2 \text{ and } y \in X_1.
\end{cases}$$

\begin{prop}\label{twisted union subquandle sep}
Let $(X_1,\star_1)$ and $(X_2,\star_2)$ be subquandle separable quandles. If $f\in \Z(\Inn(X_1))$ and $g\in \Z(\Inn(X_2))$ are finite order elements, then $(X,\ast)$ is subquandle separable.
\end{prop}

\begin{proof}
Let $Y$ be a  finitely generated subquandle of $X$ and $z\in X\setminus Y$. Without loss of generality, we can assume that $z\in X_1$. We first claim that $Y\cap X_1$ is a finitely generated subquandle of $X_1$. Suppose that $Y$ is generated by $\{a_1,a_2,\ldots,a_k,b_1,b_2,\ldots,b_\ell\}$, where each $a_i\in X_1$ and each $b_j\in X_2$. Suppose that the order of $f$ is $m$. Since $f\in \Z(\Inn(X_1))$, a direct check using the quandle operation in $X$ shows that $Y\cap X_1$ is generated by the finite set $\{f^{t_i}(a_i)\mid 1\leq i\leq k, ~ 1\leq t_i\leq m\}$.  Since $z\in X_1$ and $Y\cap X_1$ is a finitely generated separable subquandle of $X_1$,  there is a surjective quandle homomorphism $\phi:X_1\to F$, where $(F,\circ)$ is a finite quandle, such that $\phi(z)\notin \phi(Y\cap X_1)$. Let $p$ be a symbol disjoint from $F$ and $F'=F\sqcup\{p\}$. Define a binary operation $\circ'$ on $F'$ as 
$$x\circ' y=\begin{cases}
			x\circ y \quad\text{ if } x,y\in F,\\
			p   \quad\quad\quad\text{ if } x=p,\\
			\phi(f(z)) \quad\text{ if } y=p,~x \ne p
		\end{cases}$$
where $z \in X_1$ is such that $\phi(z)=x$. We claim that $(F',\circ')$ is a quandle. Suppose that $z_1, z_2 \in X_1$ such that $\phi(z_1)=\phi(z_2)=x$. Since $f \in \Inn(X_1)$,
we see that  $\phi(f(z_1))= \phi(f(z_2))$, and the binary operation $\circ'$ is indeed well-defined. Let $S_p: F' \to F'$ be the right multiplication by $p$. For arbitrary $x_1, x_2 \in F$, let $z_1, z_2 \in X_1$ such that $\phi(z_1)=x_1$ and  $\phi(z_2)=x_2$.  Suppose that $S_p(x_1)=S_p(x_2)$, that is, $x_1 \circ' p= x_2 \circ' p$. This gives $\phi(f(z_1))= \phi(f(z_2))$. Since $f \in \Inn(X_1)$, it follows that $x_1=\phi(z_1)=\phi(z_2)=x_2$. Thus, $S_p$ is injective, and finiteness of $F'$ implies that it is a bijection of $F'$. Further, we see that 
$$S_p(x_1 \circ' x_2)= S_p(x_1 \circ x_2)= \phi f (z_1 \star_1 z_2) =\phi f(z_1) \circ' \phi f(z_2) = S_p(x_1) \circ' S_p(x_2),$$
$$S_p(x_1 \circ' p)= S_p(\phi f(z_1))= \phi f(f(z_1)) =\phi f(z_1) \circ' p = S_p(x_1) \circ' S_p(p)$$
and
$$S_p(p \circ' x_2)= S_p(p)=p =p \circ' S_p(x_2) = S_p(p) \circ' S_p(x_2).$$
Using the fact that $f \in \Z(\Inn(X_1))$, we have
$$S_{x_1}(x_2 \circ' p)= S_{x_1}(\phi f (z_2))=\phi S_{z_1} f(z_2)=\phi f S_{z_1}(z_2) = S_{x_1}(x_2) \circ' p = S_{x_1}(x_2) \circ' S_{x_1}(p)$$
and
$$S_{x_1}(p \circ' x_2)= S_{x_1}(p)=p= S_{x_1}(p) \circ' S_{x_1}(x_2).$$
This proves our claim. It is easy to see that the map $\Phi:X\to F'$ defined by $\Phi(X_2)=p$ and $\Phi(y)=\phi(y)$ for $y\in X_1$,  is a quandle homomorphism. Further, $\Phi(z)\notin \Phi(Y)$, and the result follows.
\end{proof}

\begin{remark}
We note that the finiteness of orders of $f$ and $g$ in Proposition \ref{twisted union subquandle sep} is used only to ascertain finite generation of $Y \cap X_1$ as a subquandle of $X_1$. If $(X_1,\star_1)$ and $(X_2,\star_2)$ are residually finite, and  $f\in \Z(\Inn(X_1))$ and $g\in \Z(\Inn(X_2))$ (not necessarily of finite orders), then the proof implies that $(X,\ast)$ is residually finite.
\end{remark}

The following observation will be used to establish subquandle separability of some abelian quandles.

\begin{prop}\label{n-abelian finite}
Let $X$ be an abelian quandle. Then the following assertions hold:
\begin{enumerate}
\item The number of connected components of $X$ equals the cardinality of a minimal generating set for $X$.
\item If $X$ is finitely generated, then its $n$-quandle $X_n$ is finite for each $n \geq 2.$
\end{enumerate}
\end{prop}
\begin{proof}
Let $S$ be a minimal generating set for $X$. Then the number of connected components of $X$ is at most $| S|$. For the converse, suppose that there exist $x,y\in S$ which are in the same connected component, that is, there is an element $\eta\in \Inn(X)$ such that $\eta(x)=y$. Since $ \Inn(X)$ is an abelian group, we can write $\eta=S_{x_1}^{\epsilon_1}S_{x_2}^{\epsilon_2}\dots S_{x_r}^{\epsilon_r}$, where $\epsilon_i\in \mathbb{Z}$ and $x_i\in S$ are distinct generators. If $x_i=x$ or $y$ for some $i$, then by reordering, we can assume that $x_1=y$ or $x_r=x$. This gives $S_{x_2}^{\epsilon_2}S_{x_3}^{\epsilon_3}\dots S_{x_{r-1}}^{\epsilon_{r-1}}(x)=y$, where none of the $x_i$ equals $x$ or $y$. Thus, the generator $y$ can be written as a product of other generators from $S$, which contradicts the minimality of $S$. This proves assertion (1).
\par
If $X$ is finitely generated and abelian, then so is $X_n$. Let $S=\{x_1,x_2,\dots,x_r\}$ be a finite generating set for $X_n$. Then, any element of $X_n$ can be written in the form $x_i\ast^{\epsilon_1}x_1\ast^{\epsilon_2}x_2\cdots\ast^{\epsilon_r}x_r$ for some $1 \le i \le r$ and  $0 \le \epsilon_j \le n-1$ with $\epsilon_i=0$. Thus, we have $|X_n|\leq rn^{r-1}$, which proves assertion (2).
\end{proof}

\begin{theorem}\label{two generated abelian quandle}
An abelian quandle generated by two elements is subquandle separable.
\end{theorem}
\begin{proof}
Let $X=\langle S\mid R\rangle$ be a presentation of $X$, where $S=\{x,y\}$.  Since $X$ is abelian, any element of $X$ can be written in the form $x*^ny$ or $y*^mx$ for some $n, m \in \mathbb{Z}$. If $X$ is finite, then there is nothing to prove. So, we assume that $X$ is infinite. We claim that $R$ must be a singleton set. Suppose that $R$ contains elements from the orbits of both $x$ and $y$. This implies that $\{x*^{n}y=x,y*^mx=y\}\subseteq R$ for some $n, m \in \mathbb{Z}$. In this case, $X$ is an $\lcm(n,m)$-quandle. It follows from Proposition \ref{n-abelian finite} that $X$ must be a finite quandle, which is a contradiction. Thus,  $R$ contains elements from only one orbit, say, $R=\{x*^{n_1}y=x, x*^{n_2}y=x, \ldots\}$. An easy calculation shows that $X=\langle S\mid x*^ny=x\rangle$, where $n=\gcd(n_1,n_2,\ldots)$, and hence the claim holds. 
\par
We can now assume that $X=\langle x,y\mid x*^n y=x\rangle$ for some $n\in\mathbb{Z}$. Let $Y$ be a finitely generated subquandle of $X$ such that $Y \ne X$. If $x*^k y,y*^l x \in Y$ for some $k, l \in \mathbb{Z}$, then $x=(x*^ky)*^{-k}(y*^lx)\in Y$ and $y=(y*^lx)*^{-l}(x*^ky)\in Y$. This gives $Y=X$, which is a contradiction. Hence, $Y$ contains elements from only one orbit,  say, that of $x$. Consequently, $Y$ is generated by $\{x*^{n_1}y, x*^{n_2}y, \dots,x*^{n_r}y\}$ for some $r \ge 1$ and $n_i \in \mathbb{Z}$. Since $X$ is abelian, it follows that $Y$ is a finite trivial subquandle. The theorem now follows from Proposition \ref{X is rf for abelian} and  Proposition \ref{finite-sub-separable}
\end{proof}

Let $X=\langle x_1, x_2,\dots,x_r\rangle$ be a finitely generated abelian quandle. Then, any element of $X$ can be written in the form 
$$x_i*^{n_1}x_1*^{n_2}x_2\dots*^{n_r}x_r$$
for some $1 \le i \le r$ and $n_j \in \mathbb{Z}$ such that $n_i=0$.  Following  \cite{MR0781361}, we denote the element $x_i*^{n_1}x_1*^{n_2}x_2\dots*^{n_r}x_r$ of $X$ by the tuple of integers $(i;n_1,n_2,\dots,n_r)$. With this notation, the quandle operation in $X$ is given by
	\begin{equation}\label{operation of abelian quandle}
	 (i;n_{11},n_{12},\dots,n_{1r})*(j;n_{21},n_{22},\dots,n_{2r})=(i;n_1,n_2,\dots,n_r),	
	\end{equation}
  where $n_k=n_{1k}$ for $k\neq j$ and $n_j=n_{1j}+1$ if $j\neq i$.

\begin{theorem}\label{free abelian subquandle sep}
A finitely generated free abelian quandle is subquandle separable.
\end{theorem}

\begin{proof}
Let $X$ be a finitely generated free abelian quandle. Then $X$ has a presentation $$X=\langle x_1, x_2,\dots,x_r\mid x_i*x_j*x_k=x_i*x_k*x_j \text{ for all } 1\leq i,j, k\leq r \rangle.$$ Let $Y$ be a finitely generated subquandle of $X$ such that $Y \ne X$. As in the proof of Theorem \ref{two generated abelian quandle}, it is clear that if $Y$ contains elements from each orbit under the action of $\Inn(X)$, then the repeated use of equation \eqref{operation of abelian quandle} implies that $Y$ contains all the generators of $X$. This gives $Y=X$, a contradiction. Now, suppose that $Y$ is generated by the set
		$$\big\{(i_1;n_{11}^1,n_{12}^1,\dots,n_{1r}^1), (i_1;n_{21}^1,n_{22}^1,\dots,n_{2r}^1),\dots,(i_1;n_{k_11}^1,n_{k_12}^1,\dots,n_{k_1r}^1),$$
		$$(i_2;n_{11}^2,n_{12}^2,\dots,n_{1r}^2), (i_2;n_{21}^2,n_{22}^2, \dots, n_{2r}^2),\dots,(i_2;n_{k_21}^2,n_{k_22}^2,\dots,n_{k_2r}^2), \dots,$$
		$$ (i_p;n_{11}^p,n_{12}^p,\dots,n_{1r}^p), (i_p;n_{21}^p,n_{22}^p,\dots,n_{2r}^p),\dots,(i_p;n_{k_p1}^p,n_{k_p2}^p,\dots,n_{k_pr}^p)\big\}.$$ 
Let us set $\{j_1, j_2, \ldots, j_q\} = \{1, 2, \ldots, r \} \setminus \{i_1, i_2, \ldots, i_p\}$. Let $x=(i;n_1,n_2,\dots,n_r)\in X\setminus Y$. If $i\in \{j_1,j_2,\dots,j_q\}$, then define the map $\eta: X\to \{a,b\}$ by $\eta(x_t)=a$ for $t\neq i$ and $\eta(x_i)=b$, where $\{a,b\}$ is a two element trivial quandle. Then, $\eta$ is a quandle homomorphism with $\eta(x)\notin \eta(Y)$, and we are done. Next, let $i\in \{i_1,i_2,\dots,i_p\}$, say $i=i_t$. Since $X$ is free abelian,  $x=(i;n_1,n_2,\dots,n_r)\notin Y$ if and only if $$(n_{j_1},n_{j_2},\dots,n_{j_q})\notin \{ (n_{1{j_1}}^{i_t},n_{1{j_2}}^{i_t},\dots,n_{1{j_q}}^{i_t}), (n_{2{j_1}}^{i_t},n_{2{j_2}}^{i_t},\dots,n_{2{j_q}}^{i_t}),\dots,(n_{k_t{j_1}}^{i_t},n_{k_t{j_2}}^{i_t},\dots,n_{k_t{j_q}}^{i_t})\}.$$ We choose a sufficiently large $N\in \mathbb{N}$ such that 
\begin{small}
$$(n_{j_1},n_{j_2},\dots,n_{j_q})\notin \{ (n_{1{j_1}}^{i_t},n_{1{j_2}}^{i_t},\dots,n_{1{j_q}}^{i_t}), (n_{2{j_1}}^{i_t},n_{2{j_2}}^{i_t},\dots,n_{2{j_q}}^{i_t}),\dots,(n_{k_t{j_1}}^{i_t},n_{k_t{j_2}}^{i_t},\dots,n_{k_t{j_q}}^{i_t})\} \mod N.$$ 
\end{small}
Let $X_N$ be the corresponding $N$-quandle of $X$, which is finite by Proposition \ref{n-abelian finite}. Then,  the quandle homomorphism $\eta:X\to X_N$ has the property that $\eta(x) \not\in \eta(Y)$, and the proof is complete.
\end{proof}

\begin{remark}
It is interesting to find an explicit example of a finitely presented quandle that is residually finite but not subquandle separable.
\end{remark}

\begin{ack}
NKD acknowledges support from the NBHM via the grant number 0204/1/2023/R \&D-II/1792. DS thanks IISER Mohali for the PhD research fellowship. MS is supported by the Swarna Jayanti Fellowship grants DST/SJF/MSA-02/2018-19 and SB/SJF/2019-20/04. The authors are grateful to the anonymous referees for numerous suggestions that have improved the paper substantially.
\end{ack}

\section{Declaration}
The authors declare that they have no conflicts of interest and that there is no data associated with this paper.

\end{document}